\documentclass{amsart}

\newcommand{\Q}{\mathbb{Q}}
\newcommand{\C}{\mathbb{C}}

\newcommand{\R}{\mathbb{R}}
\newcommand{\N}{\mathbb{N}}
\newcommand{\boldB}{\mathbf{B}}
\newcommand{\F}{\mathbb{F}}

\newcommand{\dom}{\operatorname{dom}}
\newcommand{\ran}{\operatorname{ran}}

\newcommand{\norm}[1]{\left\| #1 \right\|}

\newcommand{\B}{\mathcal{B}}
\newcommand{\strb}{\mathcal{B}}
\newcommand{\zerovec}{\mathbf{0}}

\newcommand{\vecx}{\mathbf{x}}
\newcommand{\sgn}{\operatorname{sgn}}

\theoremstyle{theorem}
\newtheorem{theorem}{Theorem}[section]
\newtheorem{lemma}[theorem]{Lemma}

\theoremstyle{definition}
\newtheorem{definition}[theorem]{Definition}

\theoremstyle{theorem}
\newtheorem{corollary}[theorem]{Corollary}

\theoremstyle{theorem}
\newtheorem{proposition}[theorem]{Proposition}

\theoremstyle{theorem}

\theoremstyle{theorem}

\theoremstyle{definition}

\theoremstyle{theorem}

\numberwithin{equation}{section}

\begin{document}
\title{Computing the exponent of a Lebesgue space}
\author{Timothy H. McNicholl}
\address{Department of Mathematics\\
Iowa State University\\
Ames, Iowa 50011}
\email{mcnichol@iastate.edu}

\begin{abstract}
We consider the question as to whether the exponent of a computably presentable Lebesgue space whose dimension is at least 2 must be computable.  We show this very natural conjecture is true when the exponent is at least 2 or when the space is finite-dimensional.  However, we also show there is no uniform solution even when given upper and lower bounds on the exponent.  The proof of this result leads to some basic results on the effective theory of stable random variables.  
\end{abstract}
\maketitle

\section{Introduction}\label{sec:intro}

Computable structure theory studies computable presentations of mathematical structures and their relationship to each other.  These investigations support the advancement of computable model theory and effective mathematics and the theory of computation generally.  
Roughly speaking, a computable presentation of a structure is a way of defining computation on the structure.
In the case of countable structures, this is achieved by numbering the elements of the structure in a suitable way; namely so that the induced relations and operations on the natural numbers are computable.  Recently, the field has expanded its purview by investigating metric structures such as metric spaces and Banach spaces (see, e.g. \cite{Melnikov.2013}, \cite{Melnikov.Nies.2013}, \cite{Melnikov.Ng.2014}, \cite{McNicholl.2017}, \cite{Clanin.McNicholl.Stull.2019}, \cite{Brown.McNicholl.2019}).   In the case of Banach spaces, a computable presentation is a numbering of a linearly dense sequence in such a way that the norm and the vector space operations can be computed.  A formal definition is given in Section \ref{sec:back} below.  

Computable presentations of Banach spaces have been studied at least since the seminal monograph of Pour-El and Richards \cite{Pour-El.Richards.1989}.  In that text, and in subsequent developments, much attention is paid to computing on Lebesgue spaces; that is, $L^p$ spaces for some value of $p$ (see e.g. \cite{Zhong.Zhang.1999}, \cite{Kunkle.2004}).  
This focus makes sense since such spaces are of fundamental importance for analysis and applied mathematics.  It is usually assumed that the exponent of the space is computable.  This is a natural assumption; in fact, one might even consider the specification of the exponent to be part of the presentation.  
Using the classification of separable $L^p$ spaces (see e.g. \cite{Cembranos.Mendoza.1997}) it is fairly easy to show that 
when $p$ is computable, every separable $L^p$ space has a computable presentation. 

Here, we wish to take a step back and question the necessity of assuming the computability of the exponent.  That is, we consider the question ``\it If a Lebesgue space is computably presentable, does it follow that its exponent is computable?\rm".  Stated this way, the answer is easily seen to be `no'.  For, no matter the exponent, a one-dimensional Lebesgue space is just the field of scalars.  So, we restrict our attention to spaces of dimension at least 2.  
Thus, we ask ``\it If a Lebesgue space whose dimension is at least 2 is computably presentable, does it follow that its exponent is computable?\rm".  

A partial answer to this question is given by Brown, McNicholl, and Melnikov in \cite{Brown.McNicholl.Melnikov.2019}.  Namely, it is shown that 
if a Lebesgue space is computably presentable, then its exponent is right-c.e. if it is smaller then $2$ and otherwise it is left-c.e..  Here, we strengthen this result by showing that the exponent must be computable if it is larger than $2$ or if the space is finite-dimensional.  

While we do not have a complete answer to our  question, we nevertheless present some strong evidence that the answer is likely to be `no'.  
Namely, we show that even given rational upper and lower bounds on the exponent as advice, there is no \emph{uniform} procedure for computing the exponent of a Lebesgue space from an index of one of its presentations.  

In part, the motivation for trying to answer this question is that, while it is somewhat technical, it nevertheless is a fundamental question to consider for the theory of computing on Banach spaces.  In particular, it would be very surprising to find one could compute on a Lebesgue space without knowing its exponent, for then even the norm would seem to be out of reach.  But, the larger part of the motivation is the number of interesting connections that are made between computability and a broad swath of material in functional analysis and probability.  For example, the proof of the aforementioned result by Brown, McNicholl, and Melnikov utilizes 
the modulus of convexity for a strictly convex Banach space and a method due to O. Hanner for computing this modulus for a Lebesgue space from its exponent \cite{Hanner.1956}.  The proof of the result presented here for exponents larger than 2 makes use of a syntactic characterization, due to J.L. Krivine, of the Banach spaces that contain a copy of an $L^p$ space as well as some non-embedding results due to Banach and Paley \cite{Banach.1987}, \cite{Paley.1936}, \cite{Krivine.1965}.  The non-uniformity result will lead us to consider effective aspects of the theory of stable random variables in probability theory.  This is for the sake of proving an effective version of an embedding result due to Bretagnolle, Dacunha-Castelle, and Krivine \cite{Bretagnolle.Dacunha-Castelle.Krivine.1965}.

The paper is organized as follows.  Section \ref{sec:back} summarizes background information from 
probability, functional analysis, and computable analysis.  Section \ref{sec:prelim} contains some preliminary results from computable analysis.  The result for exponents larger than 2 is then presented in 
Section \ref{sec:exp.>.2}.  In Section \ref{sec:comp.r.stable}, we leave the main track of our thought to prove the results we require from the effective theory of stable random variables.  In Section \ref{sec:comp.embed}, we demonstrate the effective version of the embedding result just mentioned, and then in Section \ref{sec:exp.index}, we demonstrate the non-uniformity result.  Finally, the proof of the result for finite-dimensional spaces is given in Section \ref{sec:fin.dim}.  In the conclusion, we summarize the results again and state a conjecture.  

\section{Background}\label{sec:back}

\subsection{Background from functional analysis}\label{sec:back::subsec:FA}

Let $\F$ denote the field of scalars, which can be $\R$ or $\C$.  Let $\F_\Q = \F \cap \Q(i)$.  
That is, $\F_\Q$ is the field of rational scalars.

Most of our arguments are not affected by the choice of scalars.  
When the field of scalars is of concern, we use the notations $L^p(\Omega; \R)$ and $L^p(\Omega;\C)$ for the real and complex $L^p$ spaces over the measure space $\Omega$.

The following was introduced in \cite{McNicholl.2019} and will later be used to build isometric embeddings of 
$L^p$ spaces.

\begin{definition}\label{def:Lp.form.disj}
Suppose $1 \leq p < \infty$, $\mathcal{B}$ is a Banach space, and $v_1, \ldots, v_n \in \mathcal{B}$. We say $v_1, \ldots, v_n$ are \emph{$L^p$-formally disjointly supported} if 
	\[
	\norm{\sum_j \alpha_j v_j}_\B^p = \sum_j |\alpha_j|^p \norm{v_j}_\B^p
	\]
	for all scalars $\alpha_1, \ldots, \alpha_n$.
We say that a sequence $\{v_n\}_{n \in \N}$ of vectors is $L^p$-formally disjointly supported if 
$v_0, \ldots, v_M$ are $L^p$-formally disjointly supported for all $M$.
\end{definition}

The choice of terminology in Definition \ref{def:Lp.form.disj} is motivated by the following two 
facts.  First, if $f_1, \ldots, f_n \in L^p(\Omega)$ are disjointly supported, then they are $L^p$-formally disjointly supported. 
By a result of J. Lamperti, if $p \neq 2$, then $L^p$-formally disjointly supported vectors in $L^p(\Omega)$ are disjointly supported \cite{Lamperti.1958}.  

Our computation of exponents larger than $2$ utilizes the following two theorems.  The first is a
result on non-embeddings, and the second is a syntactic characterization of spaces containing a 
copy of a Lebesgue space.

\begin{theorem}\label{thm:non.embed}
Suppose $2 < p_1 < p_2 < \infty$.
\begin{enumerate}
	\item No infinite-dimensional $L^{p_2}$ space isometrically embeds into an $L^{p_1}$ space. \label{thm:non.embed::itm:>}
	
	\item No infinite-dimensional $L^{p_1}$ space isometrically embeds into an $L^{p_2}$ space.\label{thm:non.embed::itm:<}
\end{enumerate}
\end{theorem}

\begin{theorem}\label{thm:krivine}
Suppose $1 \leq p < \infty$, and let $\B$ be a Banach space.  
Let $r,s$ be positive integers so that $2(r-1) < p < 2r \leq 4s$.
Then, $\B$ isometrically embeds into an $L^p$ space if and only if 
for all $v_1, \ldots, v_n \in \B$ and all $\alpha_1, \ldots, \alpha_n \in \R$, 
\[
(-1)^r \sum_{\sigma \in \{-1,1\}^{2s}} \sum_{\tau \in \{1, \ldots, n\}^{2s}} \norm{ \sum_{j < 2s} \sigma(j) v_{\tau(j)}}_{\B}^q \cdot \prod_{j < 2s} \alpha_{\tau(j)} \geq 0
\] 
provided $\sum_j \alpha_j = 1$.
\end{theorem}

Part (\ref{thm:non.embed::itm:>}) of Theorem \ref{thm:non.embed} was proven by Banach \cite{Banach.1987}.  Part (\ref{thm:non.embed::itm:<})  of Theorem \ref{thm:non.embed} was proven by Paley \cite{Paley.1936}.
Theorem \ref{thm:krivine} was proven by J.L. Krivine in 1965 \cite{Krivine.1965}.

\subsection{Background from probability: stable random variables}\label{sec:back::subsec:stable}

The theory of stable random variables was initiated by P. Levy in 1925 \cite{Levy.1925} and further formalized in his 1951 monograph \cite{Levy.1954}.  Applications of stable random variables actually began in 1919 with Holtsmark's work in astronomy \cite{Holtsmark.1919}.  Afterward, they were been applied to many fields including physics, biology, economics, finance, and signal processing (see e.g. \cite{Samorodnitsky.Taqqu.1994},\cite{Zolotarev.1986} for surveys).  Today stable distributions continue to be a very active area of investigation in both pure and applied probability theory; see e.g. \cite{Uchiyama.2019}, \cite{Xu.2019}, \cite{Yu.et.al.2013}.

The material in this subsection is drawn from \cite{Samorodnitsky.Taqqu.1994} and \cite{Zolotarev.1986}.   

The following is a consequence of the proof of Theorem 1.2.2 of \cite{Durrett.2010} and will be used 
to generate random variables from distribution functions.

\begin{proposition}\label{prop:cdf.rv}
If $F$ is a cumulative distribution function, and if $g(t) = \sup F^{-1}[(0,t)]$ for all $t \in (0,1)$, 
then $F$ is the distribution function of $g$.
\end{proposition}

We write $X =_d Y$ when $X$ and $Y$ are identically distributed random variables.

A random variable $X$ is said to be \emph{symmetric} if $X =_d -X$.  

\begin{definition}\label{def:stable}
Let $X$ be a random variable.
\begin{enumerate}
	\item  $X$ is said to be \emph{stable} if for all positive $A$, $B$, there is a positive $C$ and a real $D$ so that whenever $X_1$ and $X_2$ are
independent random variables and $X_1 =_d X_2 =_d X$, $AX_1 + BX_2 =_d CX + D$.
	
	\item If, in addition, we can always choose $D$ to be zero, then $X$ is \emph{strictly stable}.
\end{enumerate}
\end{definition}

The property of stability is a property of the distribution of a random variable.  
Thus, we shall also speak of stable distributions.  Every stable distribution is characterized 
by four parameters.  The first of these is the stability index whose existence is given by the following theorem which is proved in Section VI.1 of \cite{Feller.1971}.

\begin{theorem}\label{thm:index}
Suppose $X$ is stable.  Then there is a unique real number $r \in (0, 2]$ so that 
whenever $A, B > 0$ and $X_1,X_2$ are independent copies of $X$ (i.e. $X_1 =_d X_2 =_d X$), 
$AX_1 + BX_2 =_d (A^r + B^r)^{1/r} X + D$ for some $D \in \R$.
\end{theorem}

\begin{definition}\label{def:index.stability}
If $X$ is stable, then the number $\alpha$ in Theorem \ref{thm:index} is called the 
\emph{index of stability} of $X$.  We also say $X$ is $\alpha$-stable.
\end{definition}

It is well known that the $2$-stable random variables are precisely the Gaussian random variables
and that the $1$-stable random variables are precisely the Cauchy random variables.  The aforementioned distribution of Holtsmark is $3/2$-stable.

When $X$ is a random variable, let $\phi_X$ denote the characteristic function of $X$; that is $\phi_X(t) = E[e^{iXt}]$.  Recall that two random variables are identically distributed if and only if 
they have the same characteristic function.  
Characteristic functions are the main tools for the analysis of stable random variables.   
In particular, the form of the characteristic function of a stable random variable leads to the
remaining three parameters that characterize a stable distribution.  
This is the content of the next theorem.

\begin{theorem}\label{thm:cf}
Let $X$ be a random variable. 
\begin{enumerate}
	\item If $X$ is $r$-stable, then 
there exist unique $\sigma \geq 0$, $\beta \in [-1,1]$, and $\delta \in \R$ so that 
\begin{equation}
\phi_X(t) = \left\{
\begin{array}{cc}
\exp\left(- \sigma^\alpha |t|^\alpha (1 - i \beta(\sgn(t)\tan(\frac{\pi \alpha}{2})) + i \delta t)\right) & \alpha \neq 1\\
\exp\left( -\sigma|t|(1 + i \beta \frac{2}{\pi}\sgn(t)\ln|t|) + i \delta t \right) & \alpha = 1
\end{array}
\right.  \label{eqn:cf}
\end{equation}
	\item If $\Phi_X$ has the form in Equation (\ref{eqn:cf}), then 
	$X$ is $r$-stable.
\end{enumerate}
\end{theorem}

The numbers $\sigma$, $\delta$, and $\beta$ are known as the \emph{scale}, \emph{shift}, and 
\emph{skewness} parameters of $X$ as well as its distribution.  We note that a stable random variable
is symmetric if and only if its skewness and shift parameters are $0$.  In addition, a $1$-stable 
random variable is strictly stable if and only if its skewness parameter is $0$.   On the other hand, 
when $r \neq 1$, an $r$-stable random variable is strictly stable if and only if its shift parameter is 
$0$.

For the sake of computing moments, we will need to know the asymptotic behavior of the tail distributions
of stable random variables.  These are given by the following theorem.

\begin{theorem}\label{thm:tails}
Suppose $X$ is an $r$-stable random variable, and suppose 
$\sigma$, $\beta$, and $\delta$ are its scale, skewness, and location parameters respectively. 
Then, there is a positive number $C$ so that the following hold for all sufficiently large $x$. 
\begin{enumerate}
	\item If $r = 2$, then 
	\[
	P[X < -x] = P[X > x] \leq  \frac{C}{x} \exp(-x^2 / (4\sigma^2)).
	\]\label{thm:tails::itm:r=2}
	
	\item If $r < 2$, then $\max\{P[X>x], P[X < -x]\} < C x^{-r}$. \label{thm:tails::itm:r<2}
\end{enumerate}
\end{theorem}

Part (\ref{thm:tails::itm:r=2}) is proven in \cite{Feller.1957}.  
Part (\ref{thm:tails::itm:r<2}) follows from Property 1.2.15 of \cite{Samorodnitsky.Taqqu.1994}.

We will also use complex-valued stable random variables.  The definition of stability for these random variables
is a straightforward adaptation of the definition for the real case.

\begin{definition}\label{def:stable.complex}
Suppose $X$ is a complex-valued random variable.  
\begin{enumerate}
	\item $X$ is \emph{stable} if 
for all $A,B > 0$ there exist a $C > 0$ and a $D \in \C$ so that whenever $X_1$ and $X_2$ are independent copies of $X$, $AX_1 + BX_2 =_d CX + D$.   

	\item If we can always choose $D = 0$, then we say 
$X$ is \emph{strictly stable}.  
\end{enumerate}
\end{definition}

Complex random variables also posses indices of stability.

\begin{theorem}\label{thm:stable.complex.index}
Suppose $X$ is a stable complex-valued random variable.  Then, there is a unique $r \in (0,2]$ so that 
for all $A,B > 0$, 
there is a $D \in \C$ so that whenever $X_1$ and $X_2$ are independent copies of $X$, 
$AX_1 + BX_2 =_d (A^r + B^r)^{1/r} + D$.
\end{theorem} 

Again, the number $r$ is called the index of stability of $X$.  

When dealing with complex-valued random variables, we will need a stronger property than 
symmetry.

\begin{definition}\label{def:isotropic}
A complex-valued random variable $X$ is \emph{isotropic} if $\zeta X =_d X$ for every unimodular complex number $\zeta$.
\end{definition}

The characteristic function of a complex-valued random variable $X$ is defined to be 
$\phi_X(z) = E[i\langle X, z \rangle]$.  Again, these functions characterize the distribution of $X$.  
Their forms are known, but these results will not play any role in our investigations.  
The interested reader is referred to Chapter 2 of \cite{Samorodnitsky.Taqqu.1994}.

\subsection{Background from computable analysis}\label{sec:back::subsec:CA}

Here we present the formal definitions of \emph{presentation} and \emph{computable presentation} of a Banach space.  We also define the computable vectors and sequences of a presentation and the computable maps between presentations.  Our approach is essentially that in \cite{Pour-El.Richards.1989}.  See also \cite{Weihrauch.2000}.

Let $\strb$ be a Banach space.  A \emph{presentation} of $\strb$ is a pair $(\strb, \{v_n\}_{n \in \N})$ 
where $\{v_n\}_{n \in \N}$ is a linearly dense sequence of vectors of $\strb$ (i.e. $\B$ is the closed linear span of $\{v_0, v_1, \ldots \}$).  If $\strb^\# = (\strb, \{v_n\}_{n \in \N})$ is a presentation of $\strb$, then we refer to $v_n$ as the \emph{$n$-th distinguished vector of $\strb^\#$}.  
Thus, to define a presentation of a Banach space, it suffices to specify the distinguished vectors.  

Suppose $\strb^\#$ is a presentation of a Banach space $\strb$.  A \emph{rational vector} of $\strb^\#$ is a 
rational linear combination of distinguished vectors of $\strb^\#$; i.e., a vector that can be expressed in the form $\sum_{j = 0}^n \alpha_j v_j$ where each $\alpha_j$ is a rational scalar and each $v_j$ is a distinguished vector of $\strb^\#$. 

A presentation $\B^\#$ is \emph{computable} if the norm function is computable on the rational vectors of $\B^\#$; i.e. there is an algorithm that given a $k \in \N$ and a (code of) a rational vector $v$ of $\B^\#$, 
computes a rational number $q$ so that $|q - \norm{v}_\B| < 2^{-k}$.  A code of such an algorithm is called an \emph{index} of the presentation.

Among all presentations of a Banach space $\mathcal{B}$, one may 
be designated as \emph{standard}; in this case, we will identify $\mathcal{B}$ with its standard presentation.  The standard presentation of the field of scalars $\F$ is defined by declaring $1$ to be the $n$-th distinguished vector for all $n$.
The standard presentations of $\ell^p_n$ and $\ell^p$ are defined via the standard bases.  
  The standard presentation of $L^p[0,1]$ is defined via the indicator functions of dyadic subintervals of $[0,1]$.  
  The standard presentation of $L^p(0,1)$ is defined similarly.  
  The standard presentation of $L^p((0,1)^\omega)$ is defined by declaring the $\langle n,k\rangle$-th 
  distinguished vector to be the indicator function of $(0,1)^n \times I_k \times (0,1)^\omega$ where $I_k$ is the $k$-th 
  dyadic subinterval of $(0,1)$. 
   Note that these presentations are computable if $p$ is.

Fix a presentation $\B^\#$ of a Banach space $\B$.  We say that a vector $v$ of $\B$ is a \emph{computable vector} of $\B^\#$ if there is an algorithm that given any $k \in N$ produces a rational vector $u$ of $\B^\#$ 
so that $\norm{u - v}_\B < 2^{-k}$.  A sequence $\{v_n\}_{n \in \N}$ is a \emph{computable sequence} of 
$\B^\#$ if $v_n$ is a computable vector of $\B^\#$ uniformly in $n$.

We now discuss computability properties of sets of vectors.  Let $\B^\#$ be a presentation of a Banach space $\B$, and let $X \subseteq \B$.  We say $X$ is \emph{c.e. open} if it is the union of a computable sequence of rational balls of $\B^\#$ (i.e. a sequence of rational balls whose centers and radii form computable sequences).  We say $X$ is \emph{c.e. closed} if it is closed and the set of all rational balls of $\B^\#$ that contain a point of $X$ is c.e..  We say $X$ is \emph{computably compact} if it is compact and if 
there is an algorithm that enumerates all finite sets $\{B_1, \ldots, B_n\}$ of rational balls of $\B^\#$ so that 
$X \subseteq \bigcup_j B_j$ and $B_j \cap X \neq \emptyset$ for each $j$.  

We will utilize the following which is essentially Lemma 5.2.5 of \cite{Weihrauch.2000}

\begin{proposition}\label{prop:HB}
If $X \subseteq \R^n$ is bounded and c.e. closed, and if $\R^n - X$ is c.e. open, then $X$ is computably compact.
\end{proposition}

Presentations $\mathcal{B}_0^\#$ and $\mathcal{B}_1^\#$ of Banach spaces $\mathcal{B}_0$ and $\mathcal{B}_1$ respectively induce an associated class of computable maps from 
$\mathcal{B}_0^\#$ into $\mathcal{B}_1^\#$.  Namely, a map $T : \mathcal{B}_0 \rightarrow \mathcal{B}_1$ is said to be a \emph{computable map of $\mathcal{B}_0^\#$ into $\mathcal{B}_1^\#$} if 
there is an algorithm $P$ with the following properties:
\begin{enumerate}
	\item Given a (code of a) rational ball $B_1$ of $\mathcal{B}_0^\#$ as input, if $P$ halts then it produces a rational ball $B_2$ of $\mathcal{B}_1^\#$ so that $T[B_1] \subseteq B_2$.  
	
	\item If $U$ is a neighborhood of $T(v)$, then there is a rational ball $B_1$ of $\mathcal{B}_0^\#$ so that $v \in B_1$ and given $B_1$, $P$ produces a rational ball $B_2 \subseteq U$.
\end{enumerate}
 In other words, it is possible to compute arbitrarily good approximations of $T(v)$ from sufficiently good approximations of $v$.  

The computability of a map can often be demonstrated by means of computable moduli of continuity.  
Specifically, when $T$ maps a Banach space $\B_0$ into a Banach space $\B_1$, we say 
$h : \N \rightarrow \N$ is a \emph{modulus of continuity} for $T$ if for every $k \in \N$ and all vectors 
$v_0$ and $v_1$ of $\B_0$, if $\norm{v_0 - v_1}_\B < 2^{-h(k)}$, then 
$\norm{T(v_0) - T(v_1)}_\B < 2^{-k}$.  Now, fix presentations $\B_0^\#$ and $\B_1^\#$ of $\B_0$ and $\B_1$ respectively.   It is fairly easy to show that if $T : \B_0 \rightarrow \B_1$ has a computable modulus of continuity, and if $T$ maps a computable and linearly dense sequence of $\B_0^\#$ onto a computable 
sequence of $\B_1^\#$, then $T$ is a computable map of $\B_0^\#$ into $\B_1^\#$.  
It then follows that if $T$ is linear and bounded, then $T$ is a computable map of $\mathcal{B}_0^\#$ into $\mathcal{B}_1^\#$ if and only if $T$ maps a computable and linearly dense sequence of $\B_0^\#$ onto a computable sequence of $\mathcal{B}_1^\#$.

The following, which is a fairly straightforward consequence of the definitions, is folklore.

\begin{proposition}\label{prop:comp.map.sets}
Suppose $\B_0^\#$ and $\B_1^\#$ are computable presentations of Banach spaces, and let $T$ be a computable map of 
$\B_0^\#$ into $\B_1^\#$.
\begin{enumerate}
	\item If $U$ is a c.e. open subset of $\B_1^\#$, then, $T^{-1}[U]$ is a c.e. open subset of $\B_0^\#$. 
	
	\item If $C$ is a c.e. closed subset of $\B_0^\#$, then 
$\overline{T[C]}$ is a c.e. closed subset of $\B_1^\#$.
\end{enumerate}
\end{proposition}

\section{Preliminary results from computable analysis}\label{sec:prelim}

We first state and prove a number of fairly straightforward results on the effective theory of integration.  
We then prove some preliminary results on effective presentations of Banach spaces.

\subsection{Effective integration}\label{sec:prelim::subsec:int}

For the sake of computing the densities and cumulative distribution functions of stable distributions, we will need the following lemmas concerning the effective theory of integration.   It will suffice to consider functions 
defined on the real line.  Let $m$ denote Lebesgue measure on $\R$.

\begin{lemma}\label{lm:int.fnc.comp}
Suppose $f : \R^2 \rightarrow \C$, $\psi : \R \rightarrow \R$, and $\int_\R \psi\ dm$ are computable. 
Assume further that $|f(x,y)| \leq \psi(y)$ for all $(x,y) \in \R^2$.  Then, the map
\[
x \mapsto \int_\R f(x,y) dm(y)
\]
is computable.
\end{lemma}

\begin{proof}
Let $\phi(x) = \int_\R f(x,y) dm(y)$.  For each $n \in \N$, let $g_n(x) = \int_{-n}^n f(x,y) dy$ for all $x \in R$.
Thus, $g_n$ is computable uniformly in $n$.  Since $\psi$ and $\int_\R \psi\ dm$ are computable, there is a computable $h : \N \rightarrow \N$ so that $\int_{|x| \geq n} \psi(x)\ dx < 2^{-k}$ for all $k \in \N$.  
Suppose $n \in \N$ and $n \geq h(k)$.  Then, for all $x \in \R$, 
\begin{eqnarray*}
|g_n(x) - \phi(x)| & \leq & \int_{|y| \geq n} |f(x,y)|\ dy\\
& \leq & \int_{|y| \geq n} \psi(y)\ dy < 2^{-k}.
\end{eqnarray*}
Thus, $\{g_n\}_{n \in \N}$ converges uniformly to $\phi$ with a computable modulus of convergence.  
Thus, $\phi$ is computable. 
\end{proof}

If $I$ is an interval, then a measurable function $f : I \rightarrow \C$ is said to be \emph{computably integrable} if $\int_I |f(x)|\ dx < \infty$.

\begin{lemma}\label{lm:dom.comp.int}
Suppose $\psi : \R \rightarrow \R$ is a computable and computably integrable function.  
Suppose $f : \R \rightarrow \R$ is a computable function for which there exists an $a \in \N$ so that 
$|f(x)| \leq \psi(x)$ whenever $|x| \geq a$.  Then, $f$ is computably integrable.
\end{lemma}

\begin{proof}
Let $k \in \N$ be given.  Since $\psi$ is computably integrable, it is possible to compute an 
$n_0 \in \N$ so that $n_0 \geq a$ and $\int_{|x| \geq n_0} \psi(x) dx < 2^{-(k+1)}$.  Since 
$f$ is computable, it is possible to compute a rational number $q$ so that 
$|q - \int_{|x| \leq n_0} |f(x)| dx| < 2^{-(k+1)}$.  Thus, 
\begin{eqnarray*}
|q - \int_{\R} |f|\ dm| & \leq & |q - \int_{|x| \leq n_0} |f(x)|\ dx| + \int_{|x| \geq n_0} |f(x)|\ dx\\
& \leq & 2^{-(k+1)} + \int_{|x| \geq n_0} \psi(x)\ dx < 2^{-k}.
\end{eqnarray*}
\end{proof}

\begin{lemma}\label{lm:comp.int.int.comp}
If $f : \R \rightarrow \R$ is computable and computably integrable, then $\int_{\R} f\ dm$ is computable.
\end{lemma}

\begin{proof}
Since $f$ is computable, $f^+$ and $f^-$ are computable.  Since $f$ is computably integrable,
$f^+$ and $f^-$ are computably integrable by Lemma \ref{lm:dom.comp.int}.  
It follows that $\int_{\R} f\ dm$ is computable.
\end{proof}

\begin{lemma}\label{lm:int.e.t.r}
If $r$ is a computable positive real number, then $\int_{-\infty}^\infty \exp(-|t|^r)\ dt$ is computable.
\end{lemma}

\begin{proof}
We first show that whenever $|t|$ is sufficiently large, $\exp(-|t|^r) \leq (1 + t^2)^{-1}$. 
By L'Hospital's Rule, 
\begin{eqnarray*}
\lim_{s \rightarrow \infty} \frac{\ln(1 + s^2)}{s^r} & = & \lim_{s \rightarrow \infty} \frac{2s(1 + s^2)^{-1}}{rs^{r-1}}\\
& = & \lim_{s \rightarrow \infty} \frac{2}{r(s^r + s^{r - 2})} = 0.
\end{eqnarray*}
Thus, $\ln(1 + s^2) \leq s^r$ whenever $s$ is sufficiently large.  
Therefore, $\exp(-s^r) \leq (1 + s^2)^{-1}$ whenever $s$ is sufficiently large.  

The computability of $\int_{-\infty}^\infty \exp(-|t|^r)\ dt$ now follows from Lemmas \ref{lm:dom.comp.int} and \ref{lm:comp.int.int.comp}.
\end{proof}

\begin{lemma}\label{lm:comp.int.func}
Suppose $f : \R \rightarrow \R$ is computable and computably integrable.  Then, 
\[
y \mapsto \int_{-\infty}^y f(t) dt
\]
is computable.
\end{lemma}

\begin{proof}
Set 
\[
F(y) = \int_{-\infty}^y f(t) dt.
\]
Since $f$ is computable and computably integrable, there is a computable and increasing 
$a : \N \rightarrow \N$ so that 
\[
\int_{|t| \geq a(k)} |f(t)| dt < 2^{-k}
\]
for all $k \in \N$.  \\

We first show that $F$ has a computable modulus of continuity.   To see this, for each $n \in \N$, let $F_n : [-n,n] \rightarrow \R$ be defined by 
\[
F_n(y) = \int_{-n}^y f(t) dt.
\]
Since $f$ is computable, $F_n$ is computable uniformly in $n$.  Thus, there is a uniformly 
computable sequence $\{g_n\}_{n \in \N}$ of functions so that $g_n$ is a modulus of continuity
for $F_n$ for each $n \in \N$.  Let $h(k) = g_{a(k+1)}(k+1)$.  

Since $a$ is computable and $\{g_n\}_{n \in \N}$ is uniformly computable, it follows that $h$ is computable.  
We now claim that $h$ is a modulus of continuity for $F$.   
To see this, let $y_1, y_2 \in \R$, and suppose $|y_1 - y_2| < 2^{-h(k)}$.  Without loss of generality, 
suppose $y_1 < y_2$.  Let $y_1' = \max\{-a(k+1),y_1\}$, and let 
$y_2' = \min\{y_2, a(k+1)\}$.  Then, by inspection of cases, 
\[
|F(y_1) - F(y_2)| ^ \leq \int_{-\infty}^{a(k+1)} |f(t)| dt + |F(y_2') - F(y_1')| + \int_{a(k+1)}^\infty |f(t)| dt.
\]
Since $|y_1 - y_2| < 2^{-h(k)}$, $|y_1' - y_2'| < 2^{-h(k)}$.  Thus, 
$|F(y_2') - F(y_1')| < 2^{-(k+1)}$ by definition of $h$.  Therefore, by definition of $a$, 
$|F(y_2) - F(y_1)| < 2^{-k}$.  

We now show that $\{F(q)\}_{q \in \Q}$ is computable.  Let $q \in \Q$, and let $k \in \N$.  
Compute $k' \in \N$ so that 
$-a(k'+1) < q$.  Compute $q_1 \in \Q$ so that 
\[
\left|q_1 - \int_{-a(k'+1)}^q f(t) dt \right| < 2^{-(k+1)}.
\]
Then, $|q_1 - F(q)| < 2^{-k}$.  

Thus, $F$ is computable.
\end{proof}

The following more or less effectivizes a well-known trick from measure theory, and will be used for calculating absolute moments of random variables.

\begin{lemma}\label{lm:comp.int.dist}
Let $X$ be a random variable, and suppose $\psi(x) = P[|X| \geq x]$ for all $x \geq 0$.  
If $\psi$ is computable, and if there is a computable and computably integrable $h :[0,\infty) \rightarrow \R$ so that 
$\psi(x) \leq h(x)$ for all sufficiently large $x$, then $E[|X|]$ is computable.
\end{lemma}

\begin{proof}
By a standard measure theory result, 
\[
E[|X|] = \int_0^\infty \psi(x) dx.
\]
It now follows from Lemma \ref{lm:dom.comp.int} that $E[|X|]$ is computable.
\end{proof}

\subsection{Effective presentations of Banach spaces}\label{sec:prelim::subsec:eff.pres}

Our first goal is to prove the following effective version of a well-known classical result on finite-dimensional Banach spaces.

\begin{proposition}\label{prop:fin.dim.ball}
If $\B^\#$ is a computable presentation of a finite-dimensional Banach space, then the closed unit ball of $\B$ is a computably compact set of $\B^\#$.
\end{proposition}

We will need the following lemma.

\begin{lemma}\label{lm:rational.lin.indep}
Let $n$ be a positive integer.  
Suppose $\B^\#$ is a computable presentation of a Banach space whose dimension is at least $n$.  
Then, there exist $n$ rational vectors $v_1$, $\ldots$, $v_n$ of $\B^\#$ so that 
$\{v_1, \ldots, v_n\}$ is linearly independent.
\end{lemma}

\begin{proof}
By way of contradiction suppose otherwise.  Let $m$ be the largest natural number so that there exists 
a linearly independent set of $m$ rational vectors of $\B^\#$.  Thus, $m < n$.  
Let $E$ be a linearly independent set of $m$ rational vectors of $\B^\#$, and let 
$V$ denote the linear span of $E$.  Thus, since $\dim(\B) > m$, $\B - V \neq \emptyset$.  
Since $V$ is a finite-dimensional subspace of $\B$, $V$ is closed, and so $\B - V$ is open.   
Thus, since the rational vectors of $\B^\#$ are dense in $\B$, $\B - V$ contains a rational vector 
$v$ of $\B^\#$.  Hence, $E \cup \{v\}$ is a linearly independent set of rational vectors of $\B^\#$, and 
$E \cup \{v\}$ contains $m+1$ vectors- a contradiction.
\end{proof}

\begin{proof}[Proof of Proposition \ref{prop:fin.dim.ball}]
Suppose $\B^\#$ is a computable presentation of a finite-dimensional Banach space, and let $n = \dim(B)$.  Let $\mathbf{B}$ denote the closed unit ball of $\B$. 

By Lemma \ref{lm:rational.lin.indep} , there exist $n$ rational vectors $v_1, \ldots, v_n$ of $\B^\#$ so that $\{v_1, \ldots, v_n\}$ is linearly independent.  Thus, $\{v_1, \ldots, v_n\}$ is a basis of $\B$.

For all $(\alpha_1, \ldots, \alpha_n) \in \F^n$, let $T(\alpha_1, \ldots, \alpha_n) = \alpha_1v_1 + \ldots + \alpha_n v_n$.  
Thus, $T$ is an isomorphism of $\F^n$ onto $\B$.  Since $v_1, \ldots, v_n$ are computable vectors of $\B^\#$, it follows that 
$T$ is a computable map of $\F^n$ onto $\B^\#$.

Let $S = T^{-1}$.  Since $T$ is a computable map of $\F^n$ onto $\B^\#$, $S$ is a computable map of $\B^\#$ onto $\F^n$.  
Let $V = S[\mathbf{B}]$.  Since $S$ is bounded, $V$ is bounded.

We claim that $\F^n - V$ is c.e. open.  To see this, note that since $T$ is bijective, $\F^n - V = S[\B - \mathbf{B}]$.  Since $\B - \mathbf{B}$ is the union of all rational open balls 
$B(v;r)$ of $\B^\#$ so that $\norm{v}_\B - r > 1$, $\B - \mathbf{B}$ is computably open.  By Proposition \ref{prop:comp.map.sets}, 
$\F^n - V$ is computably open.  

We now show $V$ is c.e. closed. 
By definition of $\mathbf{B}$, if $B(v;r)$ is a rational ball of $\B^\#$, then $B(v;r) \cap U \neq \emptyset$ if and only if $\norm{v}_\B - r < 1$.  
Thus, $\mathbf{B}$ is c.e. closed.  Since $V = S[\mathbf{B}]$, it now follows from Proposition \ref{prop:comp.map.sets} that $V$ is c.e. closed.  

It now follows from Proposition \ref{prop:HB} that $V$ is computably compact.
\end{proof}

We conclude this section by introducing the concept of a uniformly computably presentable sequence of 
Banach spaces.  In addition, we prove a useful result for encoding a c.e. set into such a sequence. 

\begin{definition}\label{def:seq.pres.unif.comp}
\begin{enumerate}
	\item A sequence $(\B_e^\#)_{e \in \N}$ of Banach space presentations is \emph{uniformly computable} if there is a computable $g: \N \rightarrow \N$ so that $g(e)$ is an index of $\B_e$ for all $e$.  
	
	\item A sequence $(\B_e)_{e \in \N}$ of Banach spaces is \emph{uniformly computably presentable} if there is a uniformly computable sequence $(\B_e^\#)_{e \in \N}$ of presentations.
\end{enumerate}
\end{definition}

\begin{proposition}\label{prop:pairs.banach}
Suppose $\B_0^\#$ and $\B_1^\#$ are computable presentations of Banach spaces.  
Let $A \subseteq \N$ be c.e., and for all $e \in \N$, let 
\[
\mathcal{A}_e = \left\{
\begin{array}{cc}
\B_1 & e \not \in A\\
\B_0 & e \in A\\
\end{array}
\right.
\]
If there is a computable isometric embedding of $\B_1^\#$ into $\B_0^\#$, then $(\mathcal{A}_e)_{e \in \N}$ is uniformly computably presentable.
\end{proposition}

\begin{proof}
Fix a computable isometric embedding $T$ of $\B_1^\#$ into $\B_0^\#$.  Let $v_n^j$ denote the $n$-th distinguished vector of $\B_j^\#$.  Let $(A_s)_{s \in \N}$ be a 
computable enumeration of $A$.  For all $e,n$, let 
\[
u_n^e = \left\{
\begin{array}{cc}
v_n^1 & e \not \in A\\
T(v_n^1) & e \in A - A_n\\
v_{(n)_0}^0 & e \in A_n\\
\end{array}
\right.
\]
It follows that $(u_n^e)_{n \in \N}$ is linearly dense in $\mathcal{A}_e$.  Let 
$\mathcal{A}_e^\# = (\mathcal{A}_e, (u_n^e)_{n \in \N})$.  

We show $(\mathcal{A}_e^\#)_{e \in \N}$ is uniformly computable.  By the \emph{s-m-n} Theorem, 
it suffices to show there is a computable function 
$f : \N^2 \times \F_\Q^{< \omega} \rightarrow \Q$ so that for all $e,k \in \N$ and all 
$\alpha \in \F_\Q^{< \omega}$, $|f(e,k,\alpha) - \norm{\sum_{j < |\alpha|} \alpha(j) u_j^e}_{\mathcal{A}_e} | < 2^{-k}$.

Let $\mathcal{R}_j$ denote the set of all rational vectors of $\B_j^\#$.  Since 
$\B_0^\#$ is computable, there is a computable $g : \N \times \mathcal{R}_0 \rightarrow \Q$ so that 
$|g(k,v) - \norm{v}_{\B_0}| < 2^{-k}$ for all $k \in \N$ and all $v \in \mathcal{R}_0$.  Since $T$ is 
computable, there is a computable $h : \N \times \mathcal{R}_1 \rightarrow \mathcal{R}_0$ so that 
$|\norm{T(v)}_{\B_0} - h(k,v)| < 2^{-k}$ for all $k \in \N$ and all $v \in \mathcal{R}_1$.  Let 
\[
f(e,k, \alpha) = g(k+1, h(k+1, \sum_{j < |\alpha|} \alpha(j) ( (1 - \chi_{A_j}(e)) v_j^1 + \chi_{A_j}(e) v_j^0))).
\]
By definition, $f$ is computable.  Let $e,k \in \N$, and let $\alpha \in \F_\Q^{<\omega}$.  Set:
\begin{eqnarray*}
\mathbf{x}_0 & = & \sum_{j < |\alpha|} \alpha(j) \chi_{A_j}(e) v_j^0 \\
\mathbf{x}_1 & = & \sum_{j < |\alpha|} \alpha(j) (1 - \chi_{A_j}(e)) v_j^1\\
\mathbf{x}_2 & = & h(k+1, \mathbf{x}_1)\\
\mathbf{y} & = & \sum_{j < |\alpha|} \alpha(j) u_j^e\\
q & = & g(k+1, \mathbf{x}_2 + \mathbf{x}_0)
\end{eqnarray*}
Thus, $q = f(e,k,\alpha)$.  \\

\noindent\it Case 1:\rm\ $e \not \in A$. \\

Thus, $\mathcal{A}_e = \mathcal{B}_1$, $\mathbf{x}_0 = \zerovec_{\B_0}$, and 
$\mathbf{y} = \mathbf{x}_1$.  Therefore, 
\begin{eqnarray*}
|q - \norm{\mathbf{y}}_{\mathcal{A}_e}| & = & |q - \norm{\mathbf{x}_1}_{\B_1}| \\
& = & |q - \norm{T(\mathbf{x}_1)}_{\B_0}| \\
& \leq & |q - \norm{\mathbf{x}_2}_{\B_0}| + |\norm{\mathbf{x}_2}_{\B_0} - \norm{T(\mathbf{x}_1)}_{\B_0}| \\
& < & 2^{-(k+1)} + 2^{-(k+1)} = 2^{-k}.
\end{eqnarray*}

\noindent\it Case 2:\rm\ $e \in A$. \\

Thus, $\mathcal{A}_e = \B_0$, and $\mathbf{y} = T(\mathbf{x}_1) + \mathbf{x}_0$.  
We have, 
\[
|q - \norm{T(\mathbf{x}_1) + \mathbf{x}_0}_{\mathcal{B}_0}| \leq
 |q - \norm{\mathbf{x}_2 + \mathbf{x}_0}_{\B_0}| + 
 |\norm{\mathbf{x}_2 + \mathbf{x}_0}_{\B_0} - \norm{T(\mathbf{x}_1) + \mathbf{x}_0}_{\B_0}|.
\]
By definition of $g$, $|q - \norm{\mathbf{x}_2 + \mathbf{x}_0}_{\B_0}| < 2^{-(k+1)}$.  
By definition of $h$, 
\[
|\norm{\mathbf{x}_2 + \mathbf{x}_0}_{\B_0} - \norm{T(\mathbf{x}_1) + \mathbf{x}_0}_{\B_0}| \leq 
\norm{\mathbf{x}_2 - T(\mathbf{x}_1)}_{\B_0} < 2^{-(k+1)}.
\]
So, again $|q - \norm{\mathbf{y}}_{\mathcal{A}_e}| < 2^{-k}$.
\end{proof}

\section{Exponents larger than $2$}\label{sec:exp.>.2}

\begin{theorem}\label{thm:p.geq.2}
If $2 \leq p < \infty$, and if there is a computable presentation of an $L^p$ space whose dimension is at least 2, then $p$ is computable.
\end{theorem}

\begin{proof}
Suppose $2 \leq p < \infty$, and let $\B^\#$ be a computable presentation of an $L^p$ space $\B$ whose dimension is at least $2$.  
We can assume $p$ is not an integer.  Fix integers $r,s$ so that $2 \leq 2(r-1) < p < 2r \leq 4s$.  

Let $n$ be a positive integer.  For all $(x_1, \ldots, x_n) \in \R^n$, let $T_n(x_1, \ldots, x_n) = \sum_j x_j$.  
For all $v_1, \ldots, v_n \in \B$ and all $\alpha_1, \ldots, \alpha_n, q \in \R$, let 
\[
F_n(v_1, \ldots, v_n, \alpha_1, \ldots, \alpha_n, q) = (-1)^r \sum_{\sigma \in \{-1,1\}^{2s}} \sum_{\tau \in \{1, \ldots, n\}^{2s}} \norm{ \sum_{j < 2s} \sigma(j) v_{\tau(j)}}_{\B}^q \cdot \prod_{j < 2s} \alpha_{\tau(j)}.
\]
Thus, $F_n$ is a computable map from $(\B^n)^{\#} \times \R^{n+1}$ into $\R$.
Let $\mathcal{U}_n$ denote the set of all triples $(B_1, B_2, I)$ that satisfy the following conditions.
\begin{enumerate}
	\item $B_1$ is a rational ball of $(\B^n)^\#$.
	\item $B_2$ is a rational ball of $\R^n$.
	\item $I$ is an open rational interval.
	\item $F_n[B_1 \times B_2 \times I] \subseteq (-\infty, 0)$.
	\item $B_2 \cap T_n^{-1}[\{1\}] \neq \emptyset$.
\end{enumerate}

We now let 
\begin{eqnarray*}
\mathcal{V}_n & = & \{I\ :\ \exists B_1, B_2\ (B_1, B_2, I) \in \mathcal{U}_n\}\\
V_n & = & \bigcup \mathcal{V}_n\\
\mathcal{V} & = & \bigcup \mathcal{V}_n\\
V & = & \bigcup \mathcal{V}
\end{eqnarray*}

We first claim that $V \cap [2(r-1), 2r] = [2(r-1), 2r] - \{p\}$. 
To this end, we first show that $V \cap [2(r-1), 2r] \subseteq [2(r-1), 2r] - \{p\}$. 
Let $q \in V_n \cap [2(r-1), 2r]$.  Then, $q \in I$ for some $I \in \mathcal{V}_n$.  
Since $I \in \mathcal{V}_n$, there exists $B_1$, $B_2$ so that $(B_1, B_2, I) \in \mathcal{U}_n$.  
Thus, $F_n[B_1 \times B_2 \times I] \subseteq (-\infty, 0)$ and $B_2 \cap T_n^{-1}[\{1\}] \neq \emptyset$.  
Let $(v_1, \ldots, v_n) \in B_1$, and let $(\alpha_1, \ldots, \alpha_n) \in B_2 \cap T_n^{-1}[\{1\}]$.  
Then, $F_n(v_1, \ldots, v_n, \alpha_1, \ldots, \alpha_n, q) < 0$ and $\sum_j \alpha_j = 1$.  So, by Theorem \ref{thm:krivine}, $\B$ does not isometrically embed into an $L^q$ space.  Therefore, $q \neq p$, and so 
$V \cap [2(r-1), 2r] \subseteq [2(r-1), 2r] - \{p\}$. 

Now, conversely, let $q \in [2(r-1), 2r] - \{p\}$.  Thus, by Theorem \ref{thm:non.embed}, $\B$ does not isometrically embed into an $L^q$ space.  Hence, by Theorem \ref{thm:krivine}, there exists a 
positive integer $n$, vectors $v_1, \ldots, v_n \in \B$, and $(\alpha_1, \ldots, \alpha_n) \in T_n^{-1}[\{1\}]$ so that $F_n(v_1, \ldots, v_n, \alpha_1, \ldots, \alpha_n,q) < 0$.  By continuity, there exist 
$(B_1, B_2, I) \in \mathcal{U}_n$ so that $(v_1, \ldots, v_n) \in B_1$, $(\alpha_1, \ldots, \alpha_n) \in B_2$, and $q \in I$.  Thus, $q \in V$, and so $V \cap [2(r-1), 2r] \supseteq [2(r-1), 2r] - \{p\}$.

We now claim that for every $k \in \N$, there exists $k' \in \N$, $q \in \Q$, and $I_1, \ldots, I_m \in \mathcal{V}$ so that $k' \geq k$, $2(r-1) < q - 2^{-k'} < q + 2^{-k'} < 2r$, and 
$[2(r-1), 2r] - (q - 2^{-k'}, q + 2^{-k'}) \subseteq \bigcup_j I_j$.  
Let $k \in \N$.  There exists $q \in \Q$ and $k' \in \N$ so that 
$k' \geq k$ and $p \in (q - 2^{-k'}, q+ 2^{-k'})\subseteq (2(r-1), 2r)$.  By our first claim, if $s \in [2(r-1), q - 2^{-k'}] \cup [q + 2^{-k'}, 2r]$, then $s \in I$ for some $I \in \mathcal{V}$.  By compactness, there exist 
$I_1, \ldots, I_m \in \mathcal{V}$ so that $[2(r-1), q - 2^{-k'}] \cup [q + 2^{-k'}, 2r] \subseteq \bigcup_j I_j$.  \\

We can now demonstrate $p$ is computable.  Given $k \in \N$, wait for $k' \in \N$, $q \in \Q$, $I_1$, $\ldots$, $I_m \in \mathcal{V}$ 
so that $k' \geq k$, $2(r-1) < q - 2^{-k'} < q + 2^{-k'} < 2r$, and 
$[2(r-1), 2r] - (q - 2^{-k'}, q + 2^{-k'}) \subseteq \bigcup_j I_j$, and then output $q$.  
By our second claim, this search always terminates.  By our first claim, $p \in (q - 2^{-k'}, q+ 2^{-k'})$ and so $|q - p| < 2^{-k}$.  
\end{proof}

It follows from Lemma 6.2 of \cite{Brown.McNicholl.Melnikov.2019} that if $1 \leq p < 2$, and if there is a computable presentation of an $L^p$ space, then $p$ is right-c.e..

\section{Detour: computable stable random variables}\label{sec:comp.r.stable}

Our approach to proving that exponents of Lebesgue spaces can not be uniformly computed from presentations now requires consideration of the effective content of the basic theory of stable random variables.   We begin with a computable inversion theorem for characteristic functions.

\begin{theorem}\label{thm:comp.inv}
Every distribution whose characteristic function is computable and computably integrable has a computable density and a computable cumulative distribution function.  
\end{theorem}

\begin{proof}
Let $\phi : \R \rightarrow \C$ be computable, and 
suppose $\phi$ is the characteristic 
function of a probability measure $\mu$ on $\R$.  Let $F$ be the distribution function of $\mu$.  
For all $y \in \R$, let 
\[
f(y) = \int_{-\infty}^\infty e^{-iyt} \phi(t) dt.
\]
By Theorem 3.3.14 of \cite{Durrett.2010}, $f$ is the density of $F$.

Since $\phi$ is computably integrable, it follows from Lemma \ref{lm:int.fnc.comp} that 
$f$ is computable. 

Since $F$ is a distribution function for a probability measure on $\R$, and since $f$ is the density of $F$, $\int_\R f\ dm = 1$.  So, by Lemma \ref{lm:comp.int.func}, $F$ is computable. 
\end{proof}

\begin{corollary}\label{cor:comp.param}
If the parameters of a stable distribution are computable, then its density and distribution functions are computable.
\end{corollary}

\begin{proof}
Let $\mu$ be a stable distribution whose parameters are computable.  Then, by Theorem \ref{thm:cf}, 
the characteristic function of $\mu$ is computable, and by Lemma \ref{lm:int.e.t.r} it is computably integrable.
Thus, by Theorem \ref{thm:comp.inv}, the density and distribution function of $\mu$ are computable.  
\end{proof}

We now address existence of computable random variables with a given stable distribution. 

\begin{theorem}\label{thm:comp.stable}
If $\mu$ is a stable distribution, and if the parameters of $\mu$ are computable, then there is a computable random variable on $(0,1)$ whose distribution is $\mu$.
\end{theorem}

\begin{proof}
Let $r$ denote the stability index of $\mu$.  Let $\sigma$, $\beta$, and $\delta$ denote the scale, skewness, and shift parameters of $\mu$ respectively.  
Let $F$ be the cumulative distribution function of $\mu$, and let $f$ be the density of $\mu$.  Set $g(t) = \sup F^{-1}[(0,t)]$ for each $t \in (0,1)$. By Proposition \ref{prop:cdf.rv}, $F$ is the distribution function of $g$. 
It now remains to show that $g$ is computable.

Let 
\[
I = \left\{\begin{array}{cc}
(\delta, \infty) & \alpha < 1\ \wedge\ \beta = 1\\
(-\infty, \delta) & \alpha < 1\ \wedge\ \beta = -1\\
\R & \mbox{otherwise}\\
\end{array}
\right.
\]
 By Remark 2.2.4 of \cite{Zolotarev.1986}, the density of $F$ is positive on $I$, and so $F$ is increasing on $I$.

Let $G = F |_I$. We claim that $g = G^{-1}$.  
To see this, we first show that $\ran(G) = (0,1)$; that is $F[I] = (0,1)$.  Since $F$ is a continuous distribution function, $(0,1) \subseteq \ran(F) \subseteq [0,1]$.  
By Remark 2.2.4 of \cite{Zolotarev.1986}, the density of $F$ is $0$ on $\R - I$.
Thus, $F[\R - I] \subseteq \{0,1\}$.  Therefore, $(0,1) \subseteq F[I]$. 
Since the density of $F$ is positive on $I$, $F[I]$ does not contain $0$ or $1$.  
Thus, $F[I] = (0,1)$.  

Now, let $t \in (0,1)$.  It suffices to show $G^{-1}(t) = \sup F^{-1}[(0,t)]$.  
To this end, we first show that if $a \in F^{-1}[(0,t)]$, then $a \leq G^{-1}(t)$. 
Let $a \in F^{-1}[(0,t)]$, and by way of contradiction suppose $a > G^{-1}(t)$.  
Since $F(a) > 0$, $a \in I$ (since the density of $F$ is zero on $\R - I$).  Since $F$ is increasing on $I$, 
$F(a) > F(G^{-1}(t)) = t$- a contradiction.  Thus, $a \leq G^{-1}(t)$.  

Now we show $G^{-1}(t) = \sup F^{-1}[(0,t)]$.  By what we have just shown, $G^{-1}(t) \geq \sup F^{-1}[(0,t)]$.  By way of contradiction, suppose $G^{-1}(t) > \sup F^{-1}[(0,t)]$.  
Suppose $G^{-1}(t) > a > \sup F^{-1}[(0,t)]$.  Then, $a \in I$.   So, $t > F(a)$ and thus 
$a \in F^{-1}[(0,t)]$- a contradiction. 

Finally, we demonstrate $g$ is computable.  Since the parameters of $\mu$ are computable, by Theorem \ref{thm:cf}, 
the characteristic function of $\mu$ is computable.  Thus, $F$ is computable by Theorem \ref{thm:comp.inv}.
Therefore, since $g = G^{-1}$, $g$ is computable.  
\end{proof}

We now attend to the computability of the absolute moments of stable random variables.

\begin{theorem}\label{thm:comp.stable.moment}
Let $X$ be an $r$-stable random variable, and suppose the parameters of the distribution of $X$ are computable.
\begin{enumerate}
	\item If $r < 2$, and if $p$ is a computable real so that $0 \leq p < r$, then 
	$E[|X|^p]$ is computable. 
	
	\item If $r = 2$, then $E[|X|^p]$ is computable for every nonnegative real $p$.
\end{enumerate}
\end{theorem}

\begin{proof}
Let $F$ be the distribution function of $X$.  
Let $p$ be a computable nonnegative real.
For each $x \geq 0$, let $\psi(x) = m(\{t\ :\ |X(t)|^p \geq x\})$.   

We first claim that $\psi$ is computable.  Since the parameters of the stable distribution of $X$ are computable, 
the characteristic function of $F$ is computable and so $F$ is computable by Theorem \ref{thm:comp.inv}. 
Since $F$ is continuous, for each $a \in \R$, $m(\{t\ :\ X(t) = a\}) = 0$.  
Thus, $\psi(x) = 1 - F(x^{1/p}) + F(-x^{1/p})$ for each $x \geq 0$.  
Since $p$ is computable, it follows that $\psi$ is computable. 

We now treat the case $r =2$.  We begin by showing there is a computable and computably integrable $h : [0,\infty) \rightarrow \R$ so that $h \geq \psi$. By Theorem \ref{thm:tails}, there exist a positive integer $C$ and a positive rational number $q$ so that $P[|X| > x] < C\exp(-qx^2)x^{-1}$ for all sufficiently large $x$.
Thus, for all sufficiently large $x$, $P[|X|^p > x] < C\exp(-q x^{-2/p})x^{-1/p}$.  
However, $\exp(q x^{2/p}) > x^{2 - 1/p}$ for all sufficiently large $x$. 
Thus, there is a positive integer $N_0$ so that 
$\psi(x) \leq x^{-2}$ for all $x \geq N_0$.  For all $x \geq 0$, let 
\[
h(x) = \left\{\begin{array}{cc}
1 & x < N_0\\
N_0^2 x^{-2} & x \geq N_0\\
\end{array}
\right.
\]
Thus, $h$ is computable and computably integrable, and $\psi \leq h$. 
It now follows from Lemma \ref{lm:comp.int.dist} that $E[|X|^p]$ is computable.  

Let us now consider the case $r < 2$.  Assume $p < r$. 
Again, we proceed by showing there is a computable and computably integrable $h : [0,\infty) \rightarrow \R$ so that $h \geq \psi$. It follows from Theorem \ref{thm:tails} that there is an integer $C > 1$ so that $\psi(x) \leq Cx^{-r/p}$ for all sufficiently large 
$x$.  For all $x \geq 0$, let 
\[
h(x) = \left\{ 
\begin{array}{cc}
C & x < 1 \\
Cx^{-r/p} & x \geq 1\\
\end{array}
\right.
\]
Since $r > p$, $h$ is computably integrable.  It again follows from Lemma \ref{lm:comp.int.dist} that $E[|X|^p]$ is computable. 
\end{proof}

We note that if $r <2$, and if $X$ is $r$-stable, then $E[|X|^p] = \infty$ whenever $p \geq r$. 

From Theorem \ref{thm:comp.stable.moment}, we immediately obtain the following which will be used 
in the next section to construct an embedding.

\begin{corollary}\label{cor:Lp.R.stable}
Suppose $p,r$ are computable reals so that $1 \leq p < r \leq 2$.  
Let $\mu$ be a stable distribution whose parameters are computable and whose stability index is $r$.  
Then, 
$L^p((0,1); \R)$ contains a computable vector whose distribution is $\mu$.
\end{corollary}

We will also require a complex version of the conclusion in Corollary \ref{cor:Lp.R.stable}.

\begin{proposition}\label{prop:Lp.C.stable}
Suppose $p,r$ are computable reals so that $1 \leq p < r < 2$.  Then, 
$L^p((0,1)^3; \C)$ contains a computable vector whose distribution is isotropic and strictly $r$-stable.
\end{proposition}

\begin{proof}
Let $\tau = \cos(\frac{\pi r}{4})^{2/r}$.  By Theorem \ref{thm:comp.stable}, there is a 
computable $A : (0,1) \rightarrow \R$ so that the distribution of $A$ is symmetric $r/2$-stable with 
skewness parameter $1$ and scale $\tau$.  By the same theorem, there is a normally distributed and 
computable $G : (0,1) \rightarrow \R$ whose mean is $0$. 

For all $t_1, t_2, t_3 \in (0,1)$, let :
\begin{eqnarray*}
\overline{A}(t_1, t_2, t_3) & = & A(t_1)\\
\overline{G}_1(t_1, t_2, t_3) & = & G(t_2)\\
\overline{G}_2(t_1, t_2, t_3) & = & G(t_3)
\end{eqnarray*}
Then, let $X = \overline{A}^{1/2}(\overline{G}_1 + i \overline{G}_2)$.  

By Corollary 2.6.4 of \cite{Samorodnitsky.Taqqu.1994}, the distribution of $X$ is isotropic $r$-stable and symmetric.  Since $X$ is a computable function from $(0,1)^3$ into $\C$, it now suffices to show
that $\norm{X}_p$ is computable.  By Tonelli's Theorem, 
\[
E[|X|^p] = E[|A|^{p/2}]E[|\overline{G}_1 + i \overline{G}_2|^p].
\]
By Theorem \ref{thm:comp.stable.moment}, $E[|A|^{p/2}]$ is computable.  
So, it remains to show that $E[|\overline{G}_1 + i \overline{G}_2|^p]$ is computable.  
For all $x \geq 0$, let  
$\psi(x) = m(\{\vec{t}\ :\ |\overline{G}_1(\vec{t}) + i \overline{G}_2(\vec{t})|^p \geq x\})$.  

We claim there are an integer $C$ and a positive rational number $q$ so that 
$\psi(x) \leq \exp(-q x^{2/p})$ for all sufficiently large $x$. 
To demonstrate this, we first note that 
\begin{eqnarray*}
\{\vec{t}\ :\ |\overline{G}_1(\vec{t}) + i \overline{G}_2(\vec{t})| \geq x \} & = & \{\vec{t}\ :\ \overline{G}_1(\vec{t})^2 + i \overline{G}_2(\vec{t})^2 \geq x^{2/p} \}\\
& \subseteq & \{\vec{t}\ :\ \overline{G}_1(\vec{t})^2 \geq x^{2/p}/2 \} \cup \{\vec{t}\ :\ \overline{G}_2(\vec{t})^2 \geq x^{2/p} \}\\
\end{eqnarray*}
Since $\overline{G}_j$ and $G$ are identically distributed, 
$\psi(x) \leq 2 m(\{t\ :\ |G(t)| \geq 2^{-1/2} x^{1/p}\})$.  
The conclusion now follows from Theorem \ref{thm:tails}.  

Thus, by Lemma \ref{lm:comp.int.dist}, 
$E[|\overline{G}_1 + i \overline{G}_2|^p]$ is computable.  Hence, 
$E[|f|^p]$ is computable.  Since $p$ is computable, it follows that $\norm{f}_p$ is computable.  
Since $f$ and $\norm{f}_p$ are computable, $f$ is a computable vector of 
$L^p((0,1)^3; \C)$.  
\end{proof}

\section{A computable embedding result}\label{sec:comp.embed}

We now apply the material in the previous section by proving the following.

\begin{theorem}\label{thm:comp.embed}
Suppose $p,r$ are computable reals so that $1 \leq p < r \leq 2$.  
Then, there is a computable isometric embedding of $\ell^r$ into $L^p[0,1]$.
\end{theorem}

Theorem \ref{thm:comp.embed} is an effective version of a famous theorem due to 
Bretagnolle, Dacuhna-Castelle, and Krivine \cite{Bretagnolle.Dacunha-Castelle.Krivine.1965}.
We divide the main part of the proof into the following two lemmas.

\begin{lemma}\label{lm:stable.form.disj}
Suppose $1 \leq p < r \leq 2$. 
\begin{enumerate}
	\item If $\{g_n\}_{n \in \N}$ is an independent family of symmetric and strictly $r$-stable random variables in 
$L^p(\Omega; \R)$, then $\{g_n\}_{n \in \N}$ is $L^r$-formally disjointly supported.\label{lm:stable.form.disj::itm:R}

	\item If $\{g_n\}_{n \in \N}$ is an independent family of isotropic and strictly $r$-stable complex random variables in $L^p(\Omega; \C)$, then $\{g_n\}_{n \in \N}$ is $L^r$-formally disjointly supported.\label{lm:stable.form.disj::itm:C}
\end{enumerate}
\end{lemma}

\begin{proof}
(\ref{lm:stable.form.disj::itm:R}): 
Let $M \in \N$, and suppose $a_j \in \R$ for each $j \leq M$.  Let:
\begin{eqnarray*}
h_0 & = & \sum_{j = 0}^M a_j g_j\\
h_1 & = & \left( \sum_{j = 0}^M |a_j|^r\right)^{1/r} g_0.
\end{eqnarray*}
Since $g_j$ is symmetric, $a_j g_j$ and $|a_j| g_j$ have the same distribution.  
Since $g_0$ is strictly $r$-stable and $\{g_0, \ldots, g_M\}$ is independent, it follows that $h_0$ and $h_1$ have the same distribution.  
Thus, $\norm{h_0}_p = \norm{h_1}_p$.  Therefore, 
\begin{eqnarray*}
\norm{\sum_{j = 0}^M a_j g_j }_p^r & = & \norm{h_0}_p^r\\
& = & \norm{h_1}_p^r \\
& = & \sum_{j = 0}^M |a_j|^r \norm{g_0}_p^r\\
& = & \sum_{j = 0}^M |a_j|^r \norm{g_j}_p^r.
\end{eqnarray*}

The proof of (\ref{lm:stable.form.disj::itm:C}) is similar.
\end{proof}

\begin{lemma}\label{lm:stable.indep}
Suppose $1 \leq p < r \leq 2$. 
\begin{enumerate}
	\item There is a computable sequence $\{g_n\}_{n \in \N}$ of $L^p((0,1)^\omega; \R)$ 
so that each $g_n$ is symmetric strictly $r$-stable and so that 
$\{g_n\}_{n \in \N}$ is independent.  \label{lm:stable.indep::itm:R}

	\item There is a computable sequence $\{g_n\}_{n \in \N}$ of $L^p((0,1)^\omega; \C)$ 
so that each $g_n$ is isotropic strictly $r$-stable and so that 
$\{g_n\}_{n \in \N}$ is independent.  \label{lm:stable.indep::itm:C}
\end{enumerate}
\end{lemma}

\begin{proof}
By Corollary \ref{cor:Lp.R.stable}, there is a computable vector $h_0$ of $L^p((0,1);\R)$ 
whose distribution is symmetric strictly $r$-stable.  By Proposition \ref{prop:Lp.C.stable}, 
there is a computable vector $h_1$ of $L^p((0,1)^3; \C)$ that is isotropic strictly $r$-stable.  

Let 
\[
g = \left\{
\begin{array}{cc}
h_0 & \F = \R\\
h_1 & \F = \C\\
\end{array}
\right.
\]
For each $f \in (0,1)^\omega$, let 
\[
g_n(f) = \left\{
\begin{array}{cc}
g(f(n)) & \F = \R\\
g(f(3n),f(3n+1), f(3n+2)) & \F = \C
\end{array}
\right.
\]
By construction, $g_n$ and $g$ have the same distribution, and $\{g_n\}_{n \in \N}$ is an independent sequence of random variables.  Thus, if $\F = \R$, then $g_n$ is symmetric strictly $r$-stable, and if 
$\F = \C$, then $g_n$ is isotropic $r$-stable.

We now show $\{g_n\}_{n \in \N}$ is a computable sequence of $L^p((0,1)^\omega)$.  By construction, $g_n$ is a computable function uniformly in $\N$.  
Since $g_n$ and $g$ have the same distribution, $\norm{g_n}_p = \norm{g}_p$.  
Since $\norm{h_0}_p$ and $\norm{h_1}_p$ are computable, 
$\norm{g}_p$ is computable.  Thus, $\{g_n\}_{n \in \N}$ is a computable sequence of $L^p((0,1)^\omega)$. 
\end{proof}

\begin{proof}[Proof of Theorem \ref{thm:comp.embed}]
By Lemmas \ref{lm:stable.indep} and \ref{lm:stable.form.disj}, there is a computable sequence $\{g_n\}_{n \in \N}$ of $L^p((0,1)^\omega)$ that is $L^r$-formally disjointly supported and so that each $g_n$ is nonzero.  
Let $h_n = \norm{g_n}_p^{-1} g_n$.  \\

We first demonstrate that for each $f \in \ell^r$, $\sum_{n = 0}^\infty f(n) h_n$ converges in the $L^p$-norm 
and $\norm{\sum_{n = 0}^\infty f(n) h_n}_p = \norm{f}_r$. Let $f \in \ell^r$.  If $k < m$, then since $\{g_n\}_{n \in \N}$ is $L^r$-formally disjointly supported, 
\[
\norm{\sum_{n = k}^m f(n) h_n}_p^r = \sum_{n = k}^m |f(n)|^r.
\]
Since $f \in \ell^r$, it follows that the partial sums of $\sum_{n = 0}^\infty f(n) h_n$ form a Cauchy 
sequence in the $L^p$-norm.  Since $L^p$ spaces are complete, $\sum_{n = 0}^\infty f(n) h_n$ converges in the $L^p$-norm.  It also follows that $\norm{\sum_{n = 0}^\infty f(n) h_n}_p = \norm{f}_r$.  

For each $f \in \ell^r$, let $T_1(f) = \sum_n f(n) h_n$.  By definition, $T_1$ is linear, and by what has just been shown, $T_1$ is an isometry. By definition, $T_1$ maps the standard basis of $\ell^r$ onto a computable sequence 
of $L^p((0,1)^\omega)$.  Thus, $T_1$ is computable. 

By Theorem 1.1 of \cite{Clanin.McNicholl.Stull.2019}, 
there is a computable isometric isomorphism of $L^p((0,1)^\omega)$ onto $L^p[0,1]$.
Let $T = T_2 \circ T_1$.  Thus, $T$ is a computable isometric embedding of 
$\ell^r$ into $L^p[0,1]$.  
\end{proof}

\section{Computing exponents from an index with advice}\label{sec:exp.index}

We are now ready to state and prove our main negative result.

\begin{theorem}\label{thm:no.uniform}
Suppose $I \subset [1,2]$ is an open interval.  There is no computable 
$f : \subseteq \N \rightarrow \N$ so that for all $e \in \N$, if $e$ is the index of a 
presentation of a Lebesgue space $\B$ whose index lies in $I$, then 
$f(e)$ is an index of the exponent of $\B$.  
\end{theorem}

\begin{proof}
By way of contradiction, suppose such a function exists.  Fix rational numbers $r_1, r_2 \in I$ so that $r_1 < r_2$.  
Let $k_0$ be the least natural number so that $r_2 - 2^{-k_0} > r_1$.  Fix a computable surjection $\nu_\Q$ of $\N$ onto $\Q$.  Let $S$ denote the set of 
all $e \in \dom(f)$ so that $k_0 \in \dom(\phi_{f(e)})$ and 
$|\nu_\Q(\phi_{f(e)}(k_0)) - r_2| < 2^{-k_0}$.  Thus, $S$ is c.e..  

Let 
\[
\B_e = \left\{
\begin{array}{cc}
L^{r_1}[0,1] & e \in S\\
\ell^{r_2} & e \not \in S\\
\end{array}
\right.
\] 
By Proposition \ref{prop:pairs.banach}, $(\B_e)_{e \in \N}$ is uniformly computably presentable.   
Thus, there is a computable $g : \N \rightarrow \N$ so that $g(e)$ indexes a presentation of 
$\B_e$ for all $e \in \ N$.  By the Recursion Theorem, there is an $e_0 \in \N$ so that 
$\phi_{g(e_0)} = \phi_{e_0}$.  Thus, $e_0$ indexes $\B_{e_0}$.  Therefore, 
$f(e_0)$ is defined.  Since $e_0 \in \dom(f)$, if $e_0 \not \in S$, then $f(e_0)$ does not index $r_2$.  
Thus, $\B_{e_0} = L^{r_1}[0,1]$, and $e_0 \in S$.  But, then $e_0$ does not index $r_1$- a contradiction.
\end{proof}

\section{The finite-dimensional case}\label{sec:fin.dim}

Although there is no uniform solution of our problem, we can take some comfort in the fact that in the finite-dimensional case, the exponent of a Lebesgue space \emph{can} be computed from one of its presentations.

\begin{theorem}\label{thm:lpn}
If $n \geq 2$, and if $\ell^p_n$ is computably presentable where $1 \leq p < \infty$, then $p$ is computable.
\end{theorem}

\begin{proof}
We can assume $p > 1$.  Let $(\ell^p_n)^\#$ be a computable presentation of $\ell^p_n$.  

By Lemma \ref{lm:rational.lin.indep}, there exist $n$ computable unit vectors $v_1, \ldots, v_n$ of $(\ell^p_n)^\#$ so that $\{v_1, \ldots, v_n\}$ is linearly independent.

Let $\mathcal{I}$ denote the set of all rational open intervals $I \subseteq (1, \infty)$ for which there exist
rational open balls $B_1, \ldots, B_k$ of $\ell^\infty_{n^2}$ that satisfy the following two conditions.
\begin{enumerate}
	\item $\{B_1, \ldots, B_k\}$ covers the closed unit ball of $\ell^\infty_{n^2}$.  \label{itm:balls}
	
	\item For each $j \in \{1, \ldots, k\}$, there exist rational scalars $\alpha_1^{(j)}$, $\ldots$, $\alpha_n^{(j)}$ so that 
	\[
	\norm{\sum_{t = 1}^n \alpha_t^{(j)} v_t }_p \neq \norm{\sum_{t=1}^n \alpha_t^{(j)} u_t}_r
	\]
	for all $r \in I$ and all $u_1, \ldots, u_n \in \ell^\infty_n$ so that $u_j \in B_j$ for each $j$. \label{itm:scalars}
\end{enumerate}

We first show that $\bigcup \mathcal{I} = (1, \infty) - \{p\}$.
Suppose $r \in \bigcup \mathcal{I}$.  Then, there exists $I \in \mathcal{I}$ so that 
$r \in I$.  
Thus, by definition of $\mathcal{I}$, $r \in (1, \infty)$.  Let $B_1, \ldots, B_k$ and $\{\alpha_t^{(j)}\}_{t,j}$ witness that $r \in I$.  
Since $\norm{v_j}_\infty \leq \norm{v_j}_p$, $\norm{v_j}_\infty \leq 1$.  Thus, 
by (\ref{itm:balls}), $(v_1, \ldots, v_n) \in B_j$ for some $j$.  Therefore, by (\ref{itm:scalars}), 
\[
\norm{\sum_{t = 1}^n \alpha_t^{(j)} v_t }_p \neq \norm{\sum_{t=1}^n \alpha_t^{(j)} v_t}_r.
\]
Thus, $r \neq p$.  

Conversely, suppose $r \in (1, \infty) - \{p\}$.  Let $\mathbf{B}$ denote the closed unit ball of 
$\ell^\infty_{n^2}$.  Note that $u \in \mathbf{B}$ if and only if 
there exist $u_1, \ldots, u_n \in \ell^\infty_n$ so that $u = (u_1, \ldots, u_n)$.\\

We claim that for all $\vecx \in \mathbf{B}$, there exist a rational ball 
$\boldB_{\vecx}$ of $\ell^\infty_{n^2}$, rational scalars $\beta_{\vecx, 1}$, $\ldots$, $\beta_{\vecx, n}$, and a rational open interval $I_{\vecx} \subseteq (1, \infty)$ so that 
$\vecx \in \boldB_{\vecx}$, $r \in I_{\vecx}$, and 
\[
\norm{\sum_j \beta_{\vecx, j} v_j }_p \neq \norm{\sum_j \beta_{\vecx, j} u_j }_{r'}
\]
for all $r' \in I_{\vecx}$ and all $u_1, \ldots, u_n \in \ell^\infty_n$ so that $(u_1, \ldots, u_n) \in \boldB$.

Let $\vecx \in \boldB$.  There exist $\vecx_1, \ldots, \vecx_n \in \ell^\infty_n$ so that $\vecx = (\vecx_1, \ldots, \vecx_n)$.  Since $\{v_1, \ldots, v_n\}$ is linearly independent, there is a unique linear function $T : \F^n \rightarrow \F^n$ so that $T(v_j) = \vecx_j$ for each $j \in \{1, \ldots, n\}$.
Since $r \neq p$, $T$ is not an isometry of $\ell_n^r$ into $\ell_n^p$ (otherwise, since $T$ is an endomorphism, $T$ is onto).  However, since $T$ is linear, $T$ is a continuous map of $\F^n$ into $\ell^p_n$. 
Thus, there exist rational scalars $\beta_{\vecx,1}, \ldots, \beta_{\vecx,n}$ so that 
$\norm{T(\sum_j \beta_{\vecx,j} v_j)}_p \neq \norm{\sum_j \beta_{\vecx, j} v_j}_r$.  
That is, $\norm{\sum_j \beta_{\vecx,j} \vecx_j}_p \neq \norm{\sum_j \beta_{\vecx, j} v_j}_r$. 
By continuity again, there is a rational interval $I_{\vecx} \subseteq (1, \infty)$ and a rational open ball $\boldB_{\vecx}$ as required.  

Since $\boldB$ is compact, there exist $\vecx^{(1)}, \ldots, \vecx^{(k)} \in \boldB$ so that 
$\boldB \subseteq \bigcup_j \boldB_{\vecx^{(j)}}$.  Let:
\begin{eqnarray*}
I & = & \bigcap_j I_{\vecx^{(j)}}\\
B_j & = & \boldB_{x^{(j)}}\\
\alpha_t^{(j)} & = & \beta_{\vecx^{(j)}, t}
\end{eqnarray*}
Then, $I$, $B_1$, \ldots, $B_k$, and $\{\alpha_t^{(j)}\}_{t,j}$ satisfy (\ref{itm:balls}) and (\ref{itm:scalars}).
Thus, $I \in \mathcal{I}$ and so $r \in \bigcup \mathcal{I}$.  This completes the demonstration that 
$\bigcup \mathcal{I} = (1, \infty) - \{p\}$.

We now show that $\mathcal{I}$ is c.e..  
By Proposition \ref{prop:fin.dim.ball}, condition (\ref{itm:balls}) is $\Sigma_1^0$.  Since $(\ell_n^p)^\#$ is a 
computable presentation and each $v_j$ is a rational vector of this presentation, 
it follows that condition (\ref{itm:scalars}) is $\Sigma_1^0$.  Thus, $\mathcal{I}$ is c.e..

We can now show $p$ is computable.  Fix an integer $N_0 > p$.  
Given $k \in \N$, search for a rational number $q$ and rational open intervals $I_1, \ldots, I_k \in \mathcal{I}$
so that $[q - 2^{-(k+1)}, q + 2^{-(k+1)}] \subseteq \bigcup_j I_j$, and output $q$.  
Since $\bigcup \mathcal{I} = (1, \infty) - \{p\}$, it follows that $p \in [q - 2^{-(k+1)}, q + 2^{-(k+1)}]$ and so 
$|p - q| < 2^{-k}$. 
\end{proof}

\section{Conclusion}\label{sec:conclusion}

We began our deliberations with a very natural question, ``Does the computable presentability of a Lebesgue space imply the computability of its exponent?"  We immediately see that the answer is negative for Lebesgue spaces of dimension 1.  For Lebesgue spaces of dimension at least 2, the ``obvious" answer is `yes', but this turns out to be difficult to prove even in the particular cases considered here.  
However, even though the consideration of these cases turns out to be unexpectedly difficult, their resolution and the demonstration of the non-uniformity result led to a number of surprising connections with functional analysis and probability.  

The one remaining case is that of infinite-dimensional spaces whose exponent is smaller than $2$.  
The best result so far for this case is that the exponent must be right-c.e. \cite{Brown.McNicholl.Melnikov.2019}.
This result is obtained via the modulus of uniform convexity.  There are a number of other moduli for describing geometric properties of Banach spaces, but these all lead to the exact same conclusion as does the consideration of $p$-types.  On the basis of these observations and our non-uniformity result, we conjecture that if $p$ is a right-c.e. real so that $1 \leq p \leq 2$, then every infinite-dimensional and separable $L^p$ space is computably presentable.  

\def\cprime{$'$}
\providecommand{\bysame}{\leavevmode\hbox to3em{\hrulefill}\thinspace}
\providecommand{\MR}{\relax\ifhmode\unskip\space\fi MR }
\providecommand{\MRhref}[2]{%
  \href{http://www.ams.org/mathscinet-getitem?mr=#1}{#2}
}
\providecommand{\href}[2]{#2}

%\bibliographystyle{amsplain}
%\bibliography{/Users/mcnichol/Box/research/bibliographies/computability,/Users/mcnichol/Box/research/bibliographies/analysis,/Users/mcnichol/Box/research/bibliographies/model_theory}

\end{document}